\documentclass{dmgt}

\usepackage{amsmath}
\usepackage{algorithm}
\usepackage{algorithmic}

\usepackage{tikz}
\usetikzlibrary{shapes,shadows,calc}
\usepgflibrary{arrows}
\usetikzlibrary{arrows, decorations.markings, calc, fadings, decorations.pathreplacing, patterns, decorations.pathmorphing, positioning,fit,petri,backgrounds}
\tikzset{%
     vert/.style = {rectangle, 
    path picture={\draw[line width=0.4mm]  (0,-0.07) -- (0,0.07); }, },
    }
\tikzset{%
    vertg/.style = {rectangle, 
    path picture={\draw[line width=0.4mm, gray]  (0,-0.07) -- (0,0.07); }, },
    }
\tikzset{%
    vertl/.style = {rectangle, 
    path picture={\draw[line width=0.4mm, lightgray]  (0,-0.07) -- (0,0.07); }, },
    }

\tikzset{nod2/.style={circle,draw=black,fill=black,inner sep=2.2pt}}
\tikzset{nodp/.style={circle,draw=black,pattern=north east lines,, pattern color=gray,inner sep=3pt}}
\tikzset{nodpale/.style={circle,draw=black,fill=lightgray,inner sep=3pt}}
\tikzset{nodel/.style={circle,draw=black,inner sep=2.2pt}}

\usepackage{enumerate}
\usepackage{enumitem}
\usepackage{subfigure}
\usepackage{float}
\theoremstyle{plain}
\newtheorem{fact}[theorem]{Fact}
\newtheorem*{claim*}{Claim}

\newauthor{%
Mehmet Akif Yetim}{%
M. A. Yetim}{%
S\"{u}leyman Demirel University\\
Department of Mathematics\\
Isparta, 32260, Turkey}[%
akifyetim@sdu.edu.tr]

\title{\textit{L}(\boldmath$\MakeLowercase{p},\MakeLowercase{q}$)-labeling of graphs with interval representations}[$L(\MakeLowercase{p,q})$-labeling of graphs with interval representations]

\keywords{$L(p,q)$-labeling, channel assignment, interval representation, square graph, interval graph, interval $k$-graph, permutation graph, circular-arc graph, cointerval graph, interval order, chromatic number}

\classnbr{05C12, 05C15, 05C62, 05C78, 05C85, 06A07}

\begin{document}

\begin{abstract}
We provide upper bounds on the $L(p,q)$-labeling number of graphs which have interval (or circular-arc) representations via simple greedy algorithms. We prove that there exists an $L(p,q)$-labeling with a span at most $\max\{2(p+q-1)\Delta-4q+2, (2p-1)\mu+(2q-1)\Delta-2q+1\}$ for interval $k$-graphs, $\max\{p,q\}\Delta$ for interval graphs, $3\max\{p,q\}\Delta+p$ for circular-arc graphs, $2(p+q-1)\Delta-2q+1$ for permutation graphs and $(2p-1)\Delta+(2q-1)(\mu-1)$ for cointerval graphs.  In particular, these improve existing bounds on $L(p,q)$-labeling of interval graphs and $L(2,1)$-labeling of permutation graphs. Furthermore, we provide upper bounds on the coloring of the squares of aforementioned classes.
\end{abstract}

\vspace*{-5mm}
\section{Introduction}\label{sect:intro}
All graphs considered in this paper are simple. If $G$ is a graph, $V(G)$ and $E(G)$ (or simply $V$ and $E$) denote the vertex and edge set of $G$, respectively. For a vertex $v\in V$, the set $N_G(v)=\{u: \, uv\in E\}$ is called the \emph{open neighborhood} of the vertex $v$, while the set $N_G[v]=N_G(v)\cup \{v\}$ is the \emph{closed neighborhood} of $v$. The cardinality of $N_G(v)$ is the \emph{degree} of $v$, denoted by $d_G(v)$. A vertex $v$ of $G$ is called a \emph{leaf} if $d_G(v)=1$, otherwise it is called \emph{non-leaf}. The \emph{clique number} and the \emph{maximum} degree of $G$ are denoted by $\omega(G)$  and $\Delta(G)$, respectively. When the graph $G$ is clear in the context, we simply abbreviate them to $\omega$ and $\Delta$. The complement $\overline{G}$ of a graph $G=(V,E)$ is the graph on $V$ such that two vertices are adjacent in $\overline{G}$ if and only if they are not adjacent on $G$.

We let $m_G(u,v):=|N_G(u)\cap N_G(v)|$ for any two distinct vertices $u,v\in V$, and define the \emph{multiplicity} of $G$ by $\mu(G)=\max \{m_G(u,v): \, u,v\in V\}$. Note that the parameter $\mu(G)$ is firstly introduced in \cite{civan2019} and proved to be useful especially when the gap between $\mu$ and $\Delta$ is large.

For non-negative integers $p$ and $q$, an $L(p,q)$-\emph{labeling} of a  graph $G$ is a labeling of its vertices with non-negative integers such that adjacent vertices receive labels with difference at least $p$ and the vertices at distance $2$ from each other get labels with difference at least $q$. The \emph{span} of an $L(p,q)$-labeling of $G$ is the difference between the smallest and the largest label used. $L(p,q)$-\emph{labeling number} $\lambda_{p,q}(G)$ of $G$ is the least integer $k$ such that $G$ admits an $L(p,q)$-labeling with span $k$.

The problem of $L(p,q)$-labeling is a generalization of the problem of $L(2,1)$-labeling introduced by Griggs and Yeh in \cite{Yeh} and mainly motivated by its application to radio channel assignment problem. Radio channel assignment problem basically concerns about finding a feasible assignment of frequencies to radio transmitters in order to avoid the signal interference, when the transmitters are close to each other.

In \cite{Yeh}, Griggs and Yeh proved that $\lambda_{2,1}(G)\leq \Delta^2+2\Delta$ and conjectured that $\lambda_{2,1}(G)\leq \Delta^2$ holds for every graph $G$. Although the conjecture has been confirmed for various graph classes, it is widely open in general and it became a main motivation for the most of the recent studies on the subject. The current best upper bound on $\lambda_{2,1}(G)$ is $\Delta^2+\Delta-2$ which is due to Gonçalves \cite{goncalves2008}. On the other hand, Havet \textit{et al.} \cite{havet2008} verified the conjecture asymptotically.  Griggs and Yeh's conjecture has been confirmed for many graph classes including paths, cycles, wheels, complete $k$-partite graphs \cite{Yeh}; trees \cite{chang1996, Yeh}, cographs, OSF-chordal, SF-chordal \cite{chang1996}, regular tiling \cite{bertossi2000}, chordal, unit interval \cite{sakai1994}, outerplanar, split, permutation \cite{Bodlaender}, cocomparability \cite{Calamoneri-N}, Hamiltonian cubic \cite{kang2008}, weakly chordal \cite{Cerioli} and generalized Petersen \cite{huang2012} graphs. We refer reader to \cite{Calamoneri-S} for a recent up to date survey.

The existence of a special vertex/edge ordering in graphs is used for many algorithmic purposes, including graph coloring problems. In particular, an appropriate choice of an ordering of the vertices of a graph provides optimal results when it is an input of a greedy algorithm. For instance, Panda and Goel \cite{Panda} used vertex ordering characterizations to obtain upper bounds for $L(2,1)$-labeling of dually chordal graphs and strongly orderable graphs by executing some greedy algorithms. In \cite{Calamoneri-N}, Calamoneri \textit{et al.} provided upper bounds on the $L(p,q)$-labeling of cocomparability, interval and unit interval graphs by making use of their ordering structure in not a greedy but a more direct way of labeling. In this paper, we use simple and natural greedy algorithms to produce $L(p,q)$-labeling of certain classes of graphs with interval representations. We perform constant approximation algorithms which are slightly modified versions of those proposed in \cite{Panda}. As an input of these algorithms, we use orderings of the vertices of these graphs supplied from their interval representation. 

Our main result is summarized in the following:
\begin{theorem}\label{mainthm} Let $G$ be a graph. Then,
\begin{equation*}
\lambda_{p,q}(G)\leq \begin{cases}
2(p+q-1)\Delta-2q+1,   & \textrm{if }\; G\textrm{ is an interval $k$-graph,}\\
\max\{p,q\}\Delta,   & \textrm{if }\; G\textrm{ is an interval graph,}\\
3\max\{p,q\}\Delta+p,   & \textrm{if }\; G\textrm{ is a circular-arc graph,}\\
2(p+q-1)\Delta-2q+1,   & \textrm{if }\; G\textrm{ is a permutation graph,}\\
(2p-1)\Delta+(2q-1)(\mu-1),   & \textrm{if }\; G\textrm{ is a cointerval graph.}
\end{cases}
\end{equation*} 
\end{theorem}

We divide the proof of Theorem~\ref{mainthm} into several steps by providing a detailed analysis of each graph class in the subsequent sections. Meantime, we recall that Ceroli and Posner~\cite{Cerioli} proved the Griggs and Yeh's conjecture for weakly chordal graphs by showing that $\lambda_{2,1}(G)\leq \Delta^2-\Delta+2$. Our main theorem refines this result by linearizing upper bounds in terms of $\Delta$ for interval $k$-graphs, interval, permutation and cointerval graphs, as these graphs are contained in weakly chordal graphs. Among these graphs, we provide first known linear upper bounds on $\lambda_{2,1}$ for interval $k$-graphs and cointerval graphs. To be more specific, we show that $\lambda_{2,1}(G)\leq \max\{4\Delta-2, \Delta+3\mu-1\}$ when $G$ is an interval $k$-graph and $\lambda_{2,1}(G)\leq 3\Delta+\mu-1$ if $G$ is a cointerval graph. Observe that each of these upper bounds can not exceed $4\Delta-1$. On the other hand, in the case of permutation graphs, we improve the best known upper bound $\max\{4\Delta-2, \, 5\Delta-8\}$ of Paul \textit{et al.}~\cite{Paul} to $4\Delta-1$. Note that $4\Delta-1\leq \max\{4\Delta-2, \, 5\Delta-8\}=5\Delta-8$ when $\Delta\geq 7$. 

In the general case, our main result improves the best known upper bound on $\lambda_{p,q}$ for interval graphs. We recall that Calamoneri \textit{et al.} \cite{Calamoneri-N} show that $\lambda_{p,q}(G)\linebreak \leq \max\{p,2q\}\Delta$ when $G$ is an interval graph. We also provide an upper bound on $\lambda_{p,q}$ for circular-arc graphs by making use of our result for interval graphs. Furthermore, we point out that our main result provides first known upper bounds on the number $\lambda_{p,q}$ for interval $k$-graphs, permutation graphs and cointerval graphs. Finally, we obtain as a by-product tight upper bounds on the chromatic number of squares of graphs we consider.

The rest of the paper is organized as follows: In Section~\ref{pre}, we give preliminary definitions and notations and provide a greedy $L(p,q)$-labeling algorithm which is needed in sequel. In Section~\ref{intk}, we draw attention to a generalization of interval graphs so-called interval $k$-graphs and prove there the corresponding claim of our main theorem. In Section~\ref{int}, we use an improved greedy labeling algorithm to find an $L(p,q)$-labeling of interval graphs and prove Theorem~\ref{mainthm} for interval and circular-arc graphs. Section~\ref{perm} deals with $L(p,q)$-labeling of permutation graphs by implementing the algorithm given in Section~\ref{pre}. In the last section, we observe the equivalence between cointerval graphs and comparability graphs of interval orders, and complete the proof of our main result.


\section{Preliminaries}\label{pre}
We first recall some general notions and notations needed throughout the paper, and repeat some of the definitions mentioned in the introduction more formally.

For any two vertices $u,v\in V$, the \emph{distance} $d_G(u,v)$ between $u$ and $v$ is the length of a shortest path between $u$ and $v$. We call the vertex $u$ a \emph{(first) neighbor} of $v$ if $uv\in E$ and a \emph{second neighbor} if $d_G(u,v)=2$. The square $G^2$ of the graph $G=(V,E)$ is the graph with the vertex set $V$ such that $u$ and $v$ are adjacent in $G^2$ if and only if $d_G(u,v)\leq 2$. A $k$-coloring of the graph $G$ is a labeling of its vertices with $k$ colors such that adjacent vertices receive distinct colors. $\chi(G)$ denotes the chromatic number of $G$, which is the least integer $k$ such that $G$ admits a $k$-coloring. 

If $I$ is a closed interval on the real line, we denote by $l(I)$ and $r(I)$, the left and the right endpoint of $I$, respectively. If a vertex $v\in V$ is assigned to a closed interval $I_v$ on the real line, we use notations $l(v):=l(I_v)$ and $r(v):=r(I_v)$ for the left and right endpoint of the corresponding interval. 

At this point, it is worth noting that there are three variations of the notion $L(p,q)$-labeling of graphs in the literature (see \cite{Calamoneri-S, Calamoneri-N}). Let us denote by $\lambda_{p,q}^i(G)$, the $L^i(p,q)$-labeling number of $G$, for each $i\in \{ 1,2,3\}$, and summarize them in the following way: For each $i$, an $L^i(p,q)$-labeling of a graph $G$ is a function $f_i:V\rightarrow \mathbb{Z}^+\cup\{0\}$ such that
\begin{equation*}
|f_1(u)-f_1(v)|\geq\begin{cases}
p,   & \textrm{if }\; uv\in E \\
q,   & \textrm{if }\; d_G(u,v)=2
\end{cases}
\end{equation*} 
\begin{equation*}
|f_2(u)-f_2(v)|\geq\begin{cases}
p,   & \textrm{if }\; uv\in E \\
q,   & \textrm{if }\; m_G(u,v)\geq 1
\end{cases}
\end{equation*} 
\begin{equation*}
|f_3(u)-f_3(v)|\geq\begin{cases}
p,   & \textrm{if }\; uv\in E \\
q,   & \textrm{if }\; d_G(u,v)\leq 2
\end{cases}
\end{equation*} 

Notice that $\lambda_{p,q}^1(G)\leq \lambda_{p,q}^2(G)\leq \lambda_{p,q}^3(G)$,  since any $L^i(p,q)$-labeling of $G$ yields an $L^j(p,q)$-labeling of $G$ for $1\leq j<i\leq3$. When $p\geq q$, these three definitions coincide, as they varies otherwise. Note that the condition for adjacent vertices in the third definition becomes redundant when $p<q$. It is also worth noting that there is a common and natural assumption that $p\geq q$, stemming from the assignment of frequencies. Thus we prefer to consider the common definition and set $L(p,q)=L^1(p,q)$. For graph classes under consideration,  it is also possible to produce an $L^2(p,q)$ or $L^3(p,q)$-labeling with larger spans, by modifying the algorithms we use. We also note that the definition of $L^3(p,q)$-labeling is used for interval and circular-arc graphs in \cite{Calamoneri-N}. Fortunately, our algorithm for interval graphs produces that kind of labeling as well. So we do not distinguish these two definitions of $L^1(p,q)$ and $L^3(p,q)$ in the case of interval and circular-arc graphs. 

We remark that an $L(1,1)$-labeling of a graph $G$ naturally provides a proper coloring of $G^2$. In the rest of our work, we will assume $p,q\geq 1$. In other words, we do not consider either of the cases $p=0$ or $q=0$, as $q=0$ corresponds to the proper vertex coloring problem. Since $L(p,q)$-labeling allows a vertex to be assigned with "$0$", we have the obvious equality $\chi(G^2)=\lambda_{1,1}(G)+1$.
Observe further that the lower bound $\lambda_{p,q}(G)\geq \Delta$ trivially holds for $p,q\geq 1$.

We now introduce an $L(p,q)$-labeling algorithm which is a slightly modified version of the one proposed in \cite{panda2011}. Given a linear ordering of the vertices of a graph, Algorithm~\ref{algo:alg1} labels the vertices one by one with the smallest label available such that current vertex will have a label with difference at least $p$ apart from the labels of its prelabeled neighbors and with the difference at least $q$ apart from the labels of already labeled vertices at distance two.

\begin{algorithm}[htb]
\begin{algorithmic}
\STATE{\hspace*{-3mm}\textbf{Input:} a graph $G$ with an ordering $\sigma=(v_1,v_2,\ldots,v_n)$ of its vertices;}
\STATE{\hspace*{-4.3mm} \textbf{Output:} an $L(p,q)$-labeling $f$ of $G$.}
\STATE \hspace*{-4mm} set $S:=\emptyset$,
\STATE { \hspace*{-4mm} \textbf{for all} $i$ \textbf{from} $n$ \textbf{to} $1$} 
\STATE{ Find the smallest non-negative integer $j$ such that $j\notin \{f(v)-p+1, \ldots, f(v)+p-1:  \, v\in N_G(v_i)\cap S\}\cup \{f(u)-q+1,\ldots, f(u)+q-1: \; u\in S \text{ and } d_G(v_i,u)=2\} $, }
\STATE{$f(v_i):=j$,}
\STATE{ $S:=S\cup \{v_i\}$,}
\STATE{\hspace*{-4mm} \textbf{end for}}
\STATE{\hspace*{-4mm} Output($f$)}
\end{algorithmic}
\caption{Greedy $L(p,q)$-labeling $(G,\sigma)$}
\label{algo:alg1}
\end{algorithm}

\begin{fact}
Algorithm~\ref{algo:alg1} yields an $L(p,q)$-labeling of $G$.
\end{fact}  

\begin{theorem} [\cite{panda2011}]
Algorithm~\ref{algo:alg1} can be implemented to run in $O(\Delta(|V|+|E|))$ time.
\end{theorem}

\section{Interval $k$-graphs}\label{intk}

The family of interval $k$-graphs is a relatively new class, firstly introduced in \cite{brown2002} as a generalization of (probe) interval graphs and interval bigraphs. Brown~\cite{brown} provides  a characterization of interval $k$-graphs in terms of consecutive ordering of its complete $r$-partite subgraphs. As there is no further characterization known for this class, the recent works concentrate on the possible characterization of its subclasses such as cocomparability interval $k$-graphs \cite{brown2018} and AT-free interval $k$-graphs \cite{decock2019}.

Recall that interval graphs are the intersection graphs of line segments on a straight line. A graph is called an \emph{interval graph} if its vertices can be assigned to closed intervals on the real line such that two vertices are adjacent if and only if their corresponding intervals intersect. In a similar vein, a bipartite graph is called an \emph{interval bigraph} if its vertices can be assigned to closed intervals on the real line such that two vertices from different parts are adjacent if and only if their corresponding intervals intersect. 

A graph $G$ is an \emph{interval $k$-graph} if each vertex $v\in V$ can be assigned to an ordered pair $(I_v, \kappa(v))$, where $I_v$ is a closed interval on the real line and $\kappa(v)\in \{1,2,\ldots,k\}$ such that $uv\in E$ if and only if $I_u\cap I_v\neq \emptyset$ and $\kappa(u)\neq \kappa(v)$. Such an assignment of a graph $G$ is called an \emph{interval} $k$-\emph{representation} of $G$. We may refer $C_1,\ldots,C_k$ as color classes, where $C_i:=\kappa^{-1}(i)$ for each $1\leq i\leq k$ (see Figure~\ref{fig:exintk}).

\begin{figure}[htb]
\centering     
\subfigure[]{
\begin{tikzpicture}[scale=0.48]
\node [nodp] at (-3,3) (b) [label=left:\scriptsize{$b$}] {};
\node [nodp] at (3,-3) (c)[label=right:\scriptsize{$c$}]{};
\node [nodpale] at (-3,-3) (d) [label=left:\scriptsize{$d$}]{}
	edge [] (b)
	edge [] (c);
\node [nodpale] at (3,3) (e) [label=right:\scriptsize{$e$}]{}
	edge [] (b)
	edge [] (c);
\node [nod2] at (0,0) (a) [label=left:\scriptsize{$a$}]{}
	edge [] (b)
	edge [] (c)
	edge [] (d)
	edge [] (e);			
\end{tikzpicture}
}
\hspace*{.5cm}
\subfigure[]{\
\begin{tikzpicture}[yscale=0.4,xscale=0.5]
\draw [help lines,dotted, ystep=22] (0,0) grid (11,7);
		\node [above] at (5,7) {$I_{a}$};
		\node [below] at (2,5) {$I_{b}$};
		\node [above] at (8.5,4) {$I_c$};
		\node [left] at (2,2) {$I_d$};
		\node [right] at (8,1) {$I_e$};
	    \node [right] at (11.5,7) {$C_1$};
	    \node [right] at (11.5,4.5) {$C_2$};
	    \node [right] at (11.5,1.5) {$C_3$};
		\draw [-] (0,3) -- (12,3);
		\draw [-] (0,6) -- (12,6);
		\draw [line width=1mm] (1,7) -- (9,7); 
			\node [vert] at (1,7) {};
       		 \node [vert] at (9,7) {};
 		\draw [line width=1mm] (2,2) -- (7,2);
			\node [vert] at (2,2) {};
       		 \node [vert] at (7,2) {};
		\draw [line width=1mm] (3,1) -- (8,1);
			\node [vert] at (3,1) {};
       		 \node [vert] at (8,1) {};
		\draw [line width=1mm] (0,5) -- (4,5);
			\node [vert] at (0,5) {};
       		 \node [vert] at (4,5) {};
		\draw [line width=1mm] (6,4) -- (11,4);
			\node [vert] at (6,4) {};
       		 \node [vert] at (11,4) {};
\end{tikzpicture} 
}
\caption{An interval $3$-graph and its interval $3$-representation}
\label{fig:exintk}
\end{figure}
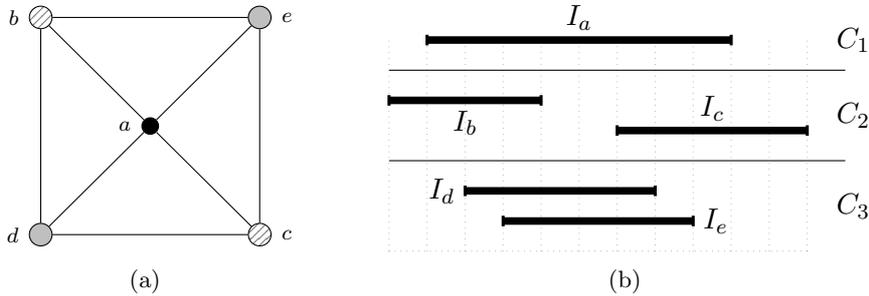

If $G$ is an interval $k$-graph, then $\chi(G)\leq k$, which is also why we will use the word color class (or interval class) for each part of $k$-partition of $G$. Interval $2$-graphs coincide with the class of interval bigraphs. On the other hand, even though every interval graph $G$ is an interval $k$-graph with $k=\chi(G)$, the converse is not true in general.

Next we present the labeling algorithm for the graphs which admit an interval representation.\\

\begin{algorithm}[htb]
\begin{algorithmic}
\STATE{\textbf{Input:} a graph $G$ with an ordering $\sigma=(v_1,v_2,\ldots,v_n)$ of its vertices such that $r(v_1)\leq r(v_2)\leq \ldots \leq r(v_n)$ in the interval representation of $G$;}
\STATE{\textbf{Output:} an $L(p,q)$-labeling $f$ of $G$.}
\STATE set $S:=\emptyset$, 
\STATE set $T:=\emptyset$,
\STATE{\textbf{\#} \it Stage 1: Label the non-leaf vertices}
\FORALL{$i$ \textbf{from} $n$ \textbf{to} $1$ } 
\IF{$d_G(v_i)=1$} 
\STATE {$T:=T\cup \{v_i\}$}
\ELSE
\STATE{ Find the smallest non-negative integer $j$ such that $j\notin \{f(v)-p+1, \ldots, f(v)+p-1:  \, v\in N_G(v_i)\cap S\}\cup \{f(u)-q+1,\ldots, f(u)+q-1: \; u\in S \text{ and } d_G(v_i,u)=2\} $, }
\STATE{$f(v_i):=j$,}
\STATE{ $S:=S\cup \{v_i\}$,}
\ENDIF
\ENDFOR
\STATE{\textbf{\#} \it Stage 2: Label the leaf vertices}
\FORALL{$v_i$ \textbf{in} $T$} 
\STATE{ Find the smallest non-negative integer $j$ such that $j\notin \{f(v)-p+1, \ldots, f(v)+p-1:  \, v\in N_G(v_i)\cap S\}\cup \{f(u)-q+1,\ldots, f(u)+q-1: \; u\in S \text{ and } d_G(v_i,u)=2\} $, }
\STATE{$f(v_i):=j$,}
\STATE{ $S:=S\cup \{v_i\}$,}
\ENDFOR
\STATE{Output($f$)}
\end{algorithmic}
\caption{Improved greedy $L(p,q)$-labeling of $G$ which admits an interval representation}
\label{algo:alg2}
\end{algorithm}

Algorithm~\ref{algo:alg2} consists of two stages. In the first stage, it labels the non-leaf vertices, starting from the vertex with the largest right endpoint until reaching the vertex with the smallest right endpoint. Therefore; at the step in which it labels the vertex $v_i$, all the vertices in $\{v_1,v_2,\ldots, v_{i-1}\}$ remains unlabeled. Thus it is sufficient to investigate those vertices having larger right endpoints among the first and the second neighbors of $v_i$, in order to compute the number of forbidden labels for $v_i$. Hence we may simply consider the vertex $v_1$ in this manner, while computing the forbidden labels. In the second stage of the algorithm, it labels the remaining unlabeled vertices each of which has degree one in $G$, in an arbitrary order.  As it can be inferred from the proof of the next theorem, the first stage of Algorithm~\ref{algo:alg2} enables us to put the parameter $\mu$ into use, by excluding the leaf-vertices.  One may easily observe that Algorithm~\ref{algo:alg2} has the same running time as Algorithm~\ref{algo:alg1}.

\begin{theorem}\label{label:intk}
If an ordering $\sigma=(v_1,v_2,\ldots,v_n)$ of the vertices of an interval $k$-graph $G$ such that $r(v_1)\leq r(v_2) \leq \ldots \leq r(v_n)$ in an interval $k$-representation of $G$ is given as the input, then Algorithm~\ref{algo:alg2} produces an $L(p,q)$-labeling of $G$ with a span at most 
\begin{equation*}
\max\{2(p+q-1)\Delta-4q+2, \, (2p-1)\mu+(2q-1)\Delta-2q+1\} \leq 2(p+q-1)\Delta-2q+1.
\end{equation*}
\end{theorem}

\begin{proof}
Let $G$ be an interval $k$-graph and $C_1,\ldots,C_k$ be its color classes. We may first observe that Algorithm~\ref{algo:alg2} produces an optimal $L(p,q)$-labeling of $G$, if $G$ consists of disjoint edges. Therefore we may assume that $\Delta(G)\geq 2$ and thus $\mu(G)\geq 1$. 

Assume that, in the first stage, the non-leaf vertices from the set $\{v_{i+1},v_{i+2},\newline \ldots,v_n\}$ have been labeled by the Algorithm~\ref{algo:alg2}. We now compute the number of forbidden labels for the vertex $v_i$. Since all the vertices from $v_1,v_2,\ldots, v_i$ are unlabeled so far, this process is equivalent to counting the forbidden labels for the vertex $v_1$ (with the assumption that $d_G(v_1)>1$), that is, the vertex with the minimum right endpoint. Now assume without loss of generality that $v_1$ belongs to the color class $C_1$. \bigskip

\textit{Case 1.} $N_G(v_1) \subseteq C_j$ for some $j\neq 1$. If all the neighbors of $v_1$ are leaf vertices, then there is no forbidden label for $v_1$, hence the algorithm assigns the label zero to $v_1$. Thus we may assume that $v_1$ has at least one non-leaf (labeled) neighbor, that is, $N_G(v_1)\cap S\neq \emptyset$. Let $z,w\in N_G(v_1)$ be the vertices such that $r(z)=\min \{r(u): \, u\in N_G(v_1)\cap S\}$ and $r(w)=\max \{r(u): \, u\in N_G(v_1)\}$ (possibly $z=w$). Choose an arbitrary neighbor $y\in N_G(z)\setminus\{v_1\}$ (which exists since $d_G(z)>1$). We first claim that all the labeled neighbors of $v_1$ are also adjacent to $y$ (note that $y$ is not adjacent to unlabeled neighbors of $v_1$, as they are leaf vertices). Recall that $r(v_1)\leq r(y)$ by the choice of $v_1$. It then follows that $[r(v_1),r(z)]\cap I_y \neq \varnothing$. On the other hand, the inclusion $[r(v_1),r(z)]\subseteq I_u$ holds for every $u\in N_G(v_1)\cap S$, by the choice of $z$. Therefore, we have $\emptyset \neq [r(v_1),r(z)]\cap I_y\subseteq I_u\cap I_y$, hence $yu\in E$ for every $u\in N_G(v_1)\cap S$, as claimed. Furthermore, all the second neighbors of $v_1$ are adjacent to $w$, because of the choice of $w$ (see Figure~\ref{fig:case1}). Thus, $v_1$ has at most $m_G(v_1,y)\leq \mu$ labeled neighbors and at most $d_G(w)-1$ second neighbors. Since for each labeled neighbor of $v_1$, there are $2(p-1)+1=2p-1$ labels that are forbidden for $v_1$, and for each labeled second neighbor of $v_1$, there are $2(q-1)+1=2q-1$ labels that are forbidden for $v_1$; totally we have at most $(2p-1)m_G(v_1,y)+(2q-1)(d_G(w)-1)\leq (2p-1)\mu+(2q-1)(\Delta-1)$ labels, which are forbidden for $v_1$.\\
\vspace{-1mm}
\begin{figure}[htb]
\centering     
\begin{tikzpicture}[yscale=0.4,xscale=0.5]
\draw [help lines,dotted, ystep=22] (0,0) grid(11,8);
		\node [above] at (0.8,4) {$I_{v_1}$};
		\node [right] at (7,0.8) {$I_w$};
		\node [right] at (4.25,2) {$I_z$};
        \node [right] at (11.5,7.5) {$C_l$};
	    \node [right] at (11.5,4.5) {$C_1$};
	    \node [right] at (11.5,1.5) {$C_j$};
		\draw [-] (0,3) -- (12,3);
		\draw [-] (0,6) -- (12,6);
		\draw [line width=1mm] (0,2) -- (4.25,2);
			\node [vert] at (0,2) {};
       		 \node [vert] at (4.25,2) {};
		\draw [line width=1mm] (0.5,4) -- (3,4);
		\draw [pattern=checkerboard, pattern color=lightgray] (0.5,3.9) rectangle (3,4.1);
			\node [vert] at (0.5,4) {};
       		 \node [vert] at (3,4) {};
		\draw [line width=1mm] (1.5,1) -- (7,1);
			\node [vert] at (1.5,1) {};
       		 \node [vert] at (7,1) {};
		\draw [line width=1mm] (1,0) -- (3.25,0); 
			\node [vert] at (1,0) {};
       		 \node [vert] at (3.25,0) {};
		\draw [line width=1mm, lightgray] (4,5) -- (5,5);
			\node [vertl] at (4,5) {};
       		 \node [vertl] at (5,5) {};
		\draw [line width=1mm, lightgray] (3.75,8) -- (5.5,8);
			\node [vertl] at (3.75,8) {};
       		 \node [vertl] at (5.5,8) {};
		\draw [line width=1mm, lightgray] (4.75,4) -- (8,4);
			\node [vertl] at (4.75,4) {};
       		 \node [vertl] at (8,4) {};
		\draw [lightgray,fill=white,pattern=north east lines, pattern color=lightgray] (6,1.9) rectangle (10.5,2.1);
			\node [vertl] at (6,2) {};
       		 \node [vertl] at (10.5,2) {};	
		\draw [line width=1mm, lightgray] (6.5,7) -- (9,7);
			\node [vertl] at (6.5,7) {};
       		 \node [vertl] at (9,7) {};
		\draw [lightgray,fill=white,pattern=north east lines, pattern color=lightgray] (7.5,7.9) rectangle (10,8.1);
			\node [vertl] at (7.5,8) {};
       		 \node [vertl] at (10,8) {};
		\draw [lightgray,fill=white,pattern=north east lines, pattern color=lightgray] (8.5,4.9) rectangle (11,5.1);
			\node [vertl] at (8.5,5) {};
       		 \node [vertl] at (11,5) {};
       \draw [line width=1mm] (15,5.5) -- (16,5.5);
       \draw [line width=1mm, lightgray] (15,3.5) -- (16,3.5);
 		\node [right] at (16,5.5) {(first) neighbors of $v_1$};
 		\node [right] at (16,3.5) {second neighbors of $v_1$};
\end{tikzpicture} 
\caption{Case 1: All the neighbors of $v_1$ belong to one particular color class.}
\label{fig:case1}
\end{figure}
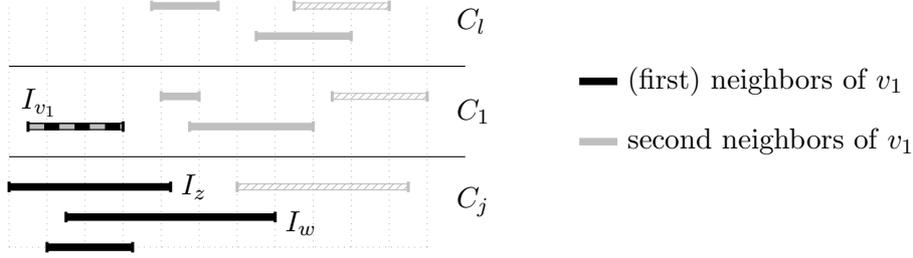
\vspace{-1mm}
\begin{figure}[htb]
\centering 
\begin{tikzpicture}[yscale=0.4,xscale=0.5]
\draw [help lines,dotted, ystep=22] (0,0) grid(11,8);
		\node [right] at (2.8,3.8) {$I_{v_1}$};
		\node [right] at (7.5,2) {$I_c$};
		\node [below] at (4.8,8) {$I_v$};
        \node [right] at (11.5,7.5) {$C_l$};
	    \node [right] at (11.5,4.5) {$C_1$};
	    \node [right] at (11.5,1.5) {$C_j$};
		\draw [-] (0,3) -- (12,3);
		\draw [-] (0,6) -- (12,6);
		\draw [line width=1mm, lightgray] (0,5) -- (5,5);
			\node [vertl] at (0,5) {};
       		 \node [vertl] at (5,5) {};
		\draw [line width=1mm] (0.5,4) -- (3,4);
		\draw [pattern=checkerboard, pattern color=lightgray] (0.5,3.9) rectangle (3,4.1);
			\node [vert] at (0.5,4) {};
       		 \node [vert] at (3,4) {};
 		\draw [line width=1mm] (1,7) -- (3.5,7);
			\node [vert] at (1,7) {};
       		 \node [vert] at (3.5,7) {};
		\draw [line width=1mm] (1.5,1) -- (4,1);
			\node [vert] at (1.5,1) {};
       		 \node [vert] at (4,1) {};
		\draw [line width=1mm] (2,8) -- (5.5,8);
			\node [vert] at (2,8) {};
       		 \node [vert] at (5.5,8) {};
		\draw [line width=1mm] (2.5,2) -- (7.5,2);
			\node [vert] at (2.5,2) {};
       		 \node [vert] at (7.5,2) {};
		\draw [line width=1mm, lightgray] (4.5,1) -- (6.5,1);
			\node [vertl] at (4.5,1) {};
       		 \node [vertl] at (6.5,1) {};
		\draw [line width=1mm, lightgray] (6,4) -- (8.5,4);
			\node [vertl] at (6,4) {};
       		 \node [vertl] at (8.5,4) {};
 		\draw [line width=1mm, lightgray] (6.5,8) -- (10.5,8);
			\node [vertl] at (6.5,8) {};
       		 \node [vertl] at (10.5,8) {};
		\draw [lightgray,fill=white,pattern=north east lines,, pattern color=lightgray] (7,0.9) rectangle (9.5,1.1);
			\node [vertl] at (7,1) {};
       		 \node [vertl] at (9.5,1) {};
		\draw [lightgray,fill=white,pattern=north east lines,, pattern color=lightgray] (8,6.9) rectangle (10,7.1);
			\node [vertl] at (8,7) {};
       		 \node [vertl] at (10,7) {};
		\draw [lightgray,fill=white,pattern=north east lines,, pattern color=lightgray] (9,4.9) rectangle (11,5.1);
			\node [vertl] at (9,5) {};
       		 \node [vertl] at (11,5) {};
       \draw [line width=1mm] (15,5.5) -- (16,5.5);
       \draw [line width=1mm, lightgray] (15,3.5) -- (16,3.5);
 		\node [right] at (16,5.5) {(first) neighbors of $v_1$};
 		\node [right] at (16,3.5) {second neighbors of $v_1$};
\end{tikzpicture} 
\caption{Case 2: $v_1$ has neighbors from at least two color classes.}
\label{fig:case2}
\end{figure}
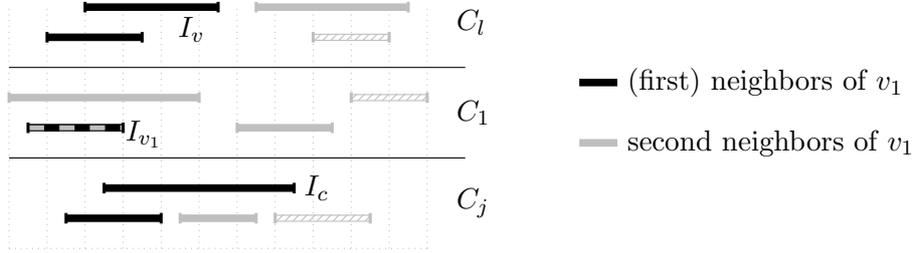
 
\textit{Case 2.} $N_G(v_1)\cap C_{\alpha}\neq \emptyset$ and $N_G(v_1)\cap C_{\beta}\neq \emptyset$ for some $\alpha,\beta \in \{2,3,\ldots,k\}$.
Let $c$ be a vertex such that $r(c)=\max \{r(u): \, u\in N_G(v_1)\}$, and suppose that $c\in C_j$ for some $j\neq 1$. Observe that $N_G(c)$ contains the first and the second neighbors of $v_1$ except for those in $C_j$. Now choose $v\in N_G(v_1)$ such that $r(v)=\max \{r(u): \, u\in N_G(v_1)\setminus C_j\}$. The set $N_G(v)$ contains the first and the second neighbors of $v_1$ in $C_j$ (see Figure~\ref{fig:case2}). Consequently, the set $(N_G(c)\setminus\{v_1\})\cup(N_G(v)\setminus\{v_1\})$ contains all the first and second neighbors of $v_1$ in $G$. If $d_G(v_1)=s$, then the number of the second neighbors of $v_1$ is at most $d_G(c)-1+d_G(v)-1-s=d_G(c)+d_G(v)-s-2$. Thus, the number of labels that are forbidden for $v_1$ is at most
\vspace{-.08cm}
\begin{equation*}
\begin{array}{rl}
&(2p-1)s+(2q-1)(d_G(c)+d_G(v)-s-2)\\
&=2ps-s+(2q-1)(d_G(c)+d_G(v))-2qs+s-4q+2 \\  
&=2(p-q)s+(2q-1)(d_G(c)+d_G(v))-4q+2  \\
&\leq 2(p-q)\Delta+(2q-1)2\Delta-4q+2 \\
&=2(p+q-1)\Delta-4q+2.
\end{array}
\end{equation*}

Let $m$ be the maximum value of the number of forbidden labels gathered from the corresponding two cases, that is, 
\begin{equation*}
m=\max \{2(p+q-1)\Delta-4q+2, (2p-1)\mu+(2q-1)\Delta-2q+1\} \leq 2(p+q-1)\Delta-2q+1. 
\end{equation*}
Then there is at least one available label from the set $\{0,1,\ldots,m\}$ in order to assign to vertex $v_1$. 

For the second stage of the algorithm, assume that $d_G(v_i)=1$. There are at most $(2p-1)+(2q-1)(\Delta-1)$ forbidden labels for the vertex $v_i$, because it has one neighbor and at most $(\Delta-1)$ vertices at distance $2$. However, we have 
\begin{align*}
(2p-1)+(2q-1)(\Delta-1)&<(2p-1)\mu+(2q-1)(\Delta-1)\\
&=(2p-1)\mu+(2q-1)\Delta-2q+1 \leq m.  
\end{align*}

This means that we always have an available label for a degree one vertex. This completes the proof.
\end{proof}

An immediate consequence of the proof of Theorem~\ref{label:intk} is that the upper bound on $\lambda_{p,q}$ becomes  $(2p-1)\mu+(2q-1)\Delta-2q+1$ for interval bigraphs, since there are only two color classes. It is also possible to achieve another upper bound for interval bigraphs in the following manner. Firstly recall that every interval bigraph is chordal bipartite. Furthermore, a strong $T$-elimination ordering of the\linebreak vertices of chordal bipartite graphs was used (in reverse) to get an $L(2,1)$-labeling with a span at most $2\Delta$ in \cite{Panda}. Using the same ordering as an input in Algorithm \ref{algo:alg1}, we may obtain an $L(p,q)$-labeling of a chordal bipartite graph with a span at most $2(2q-1)(\Delta-1)+p$.

We also have the following two corollaries as a result of Theorem~\ref{label:intk}.

\begin{cor}\label{cor:intk}
 If $G$ is an interval $k$-graph, then
\begin{equation*}
\lambda_{2,1}(G)\leq  \max\{4\Delta-2, \Delta+3\mu-1\}\leq 4\Delta-1.
\end{equation*} 
\end{cor}

Notice that Corollary~\ref{cor:intk} implies that Algorithm~\ref{algo:alg2} for $L(2,1)$-labeling of interval $k$-graphs is a $4$-approximation algorithm.

\begin{cor}
$\chi(G^2)\leq \max\{2\Delta-1,\Delta+\mu\} \leq 2\Delta $ for every interval $k$-graph $G$.
\end{cor}


\section{Interval and Circular-Arc Graphs}\label{int}
Regarding the labeling of interval graphs, the first known result is due to Chang and Kuo \cite{chang1996}. In their paper, they show that $\lambda_{2,1}(G)\leq 2\Delta$ when $G$ is an interval graph. Such a bound is generalized for an arbitrary $p\geq 2$ in~\cite{chang2000} by showing that $\lambda_{p,1}(G)\leq p\Delta$ holds. 
In fact, these results are a consequence of a more general result. The
inequality $\lambda_{p,1}(G)\leq p\Delta$ holds for odd-sun-free chordal graphs, which constitutes a superclass of interval graphs. A further generalization in the case of interval graphs appeared in \cite{Calamoneri-N}.

\begin{theorem}[\cite{Calamoneri-N}]\label{cala:int}
$\lambda_{p,q}\leq \max\{p,2q\}\Delta$ holds for every interval graph $G$.
\end{theorem}

Our next target is to improve the upper bound on $\lambda_{p,q}$ for interval graphs given in Theorem~\ref{cala:int}.

The next algorithm we implement for $L(p,q)$-labeling of interval graphs can also be applied to any graph in general. It will clearly produce an $L(p,q)$-labeling of a given graph with an arbitrary ordering of its vertices. However, unlike many graph classes, it works efficiently for particularly chosen classes as in the example of dually chordal graphs in \cite{Panda}, where the idea of the algorithm is borrowed from.

\begin{algorithm}[htb]
\begin{algorithmic}
\STATE{\textbf{Input:} a graph $G$ with an ordering $\sigma=(v_1,v_2,\ldots,v_n)$ of its vertices;}
\STATE{\textbf{Output:} an $L(p,q)$-labeling $f$ of $G$.}
\STATE set $S:=\emptyset$,
\FORALL {$i$ \textbf{from} $n$ \textbf{to} $1$}
\STATE{ Find the smallest non-negative integer $j$ such that $j\max\{p,q\} \notin \{f(v):  \, v\in N_G(v_i)\cap S\}\cup \{f(u): \; u\in S \text{ and } d_G(v_i,u)=2\} $, }
\STATE{$f(v_i):=j\max\{p,q\}$,}
\STATE{ $S:=S\cup \{v_i\}$,}
\ENDFOR
\STATE{Output($f$)}
\end{algorithmic}
\caption{Improved greedy $L(p,q)$-labeling of $(G,\sigma)$}
\label{algo:alg3}
\end{algorithm}

\begin{lemma}\label{lem:int}
Let $v$ be a vertex of an interval graph $G$ such that the interval $I_v$ has the minimum right endpoint in an interval representation of $G$. If $w$ is a neighbor of $v$ with the maximum right endpoint, then
$N_{G^2}(v) \subseteq N_G[w]\setminus \{v\}$.
\end{lemma}

\begin{proof}
Since $r(v)\in I_u$ for all $u\in N_G(v)$, it is clear that the inclusion $N_G(v) \subseteq N_G[w]\setminus \{v\}$ holds. Now, let $u$ be any vertex at distance $2$ from $v$. Then $u$ must be adjacent to some neighbor $z$ of $v$ and satisfies $l(u)>r(v)$ (see Figure~\ref{fig:int}). Then we have that $\emptyset\neq I_u\cap I_z \subseteq I_u\cap I_w$ since $r(z)\leq r(w)$.
\end{proof}

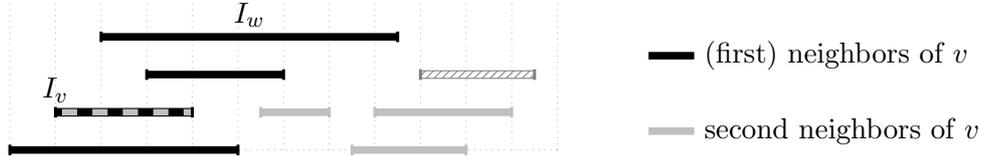
\begin{figure}[htb]
\centering 
\begin{tikzpicture}[yscale=0.5,xscale=0.6]
\draw [help lines,dotted, ystep=22] (0,0) grid(12,4);
		\node [above] at (1,1) {$I_{v}$};
		\node [above] at (5.25,3) {$I_{w}$};
		\draw [line width=1mm] (0,0) -- (5,0); 
			\node [vert] at (0,0) {};
       		 \node [vert] at (5,0) {};
		\draw [line width=1mm] (1,1) -- (4,1);
		\draw [pattern=checkerboard, pattern color=lightgray] (1,0.9) rectangle (4,1.1);
			\node [vert] at (1,1) {};
       		 \node [vert] at (4,1) {};
 		\draw [line width=1mm] (2,3) -- (8.5,3);
			\node [vert] at (2,3) {};
       		 \node [vert] at (8.5,3) {};
		\draw [line width=1mm] (3,2) -- (6,2);
			\node [vert] at (3,2) {};
       		 \node [vert] at (6,2) {};
		\draw [line width=1mm,lightgray] (5.5,1) -- (7,1);
			\node [vertl] at (5.5,1) {};
       		 \node [vertl] at (7,1) {};
		\draw [line width=1mm,lightgray] (7.5,0) -- (10,0);
			\node [vertl] at (7.5,0) {};
       		 \node [vertl] at (10,0) {};
		\draw [line width=1mm,lightgray] (8,1) -- (11,1);
			\node [vertl] at (8,1) {};
       		 \node [vertl] at (11,1) {};
		\draw [gray,fill=white,pattern=north east lines,, pattern color=gray] (9,1.9) rectangle (11.5,2.1);
			\node [vertg] at (9,2) {};
       		 \node [vertg] at (11.5,2) {};
       \draw [line width=1mm] (14,2.5) -- (15,2.5);
       \draw [line width=1mm, lightgray] (14,0.5) -- (15,0.5);
 		\node [right] at (15,2.5) {(first) neighbors of $v$};
 		\node [right] at (15,0.5) {second neighbors of $v$};
\end{tikzpicture} 
\caption{An illustration for the proof of Lemma~\ref{lem:int}.}
\label{fig:int}
\end{figure}

\begin{theorem}\label{label:int}
If an ordering $\sigma=(v_1,v_2,\ldots,v_n)$ of the vertices of an interval graph $G$ such that $r(v_1)\leq r(v_2) \leq \ldots \leq r(v_n)$ in an interval representation of $G$ is given as the input, then Algorithm~\ref{algo:alg3} finds an $L(p,q)$-labeling of $G$ with a span at most $\max\{p,q\}\Delta$.
\end{theorem}

\begin{proof}
Let $m=\max\{p,q\}$ and assume that the vertices $v_{i+1},v_{i+2},\ldots,v_n$ of $G$ has been labeled by the Algorithm~\ref{algo:alg3}. Similar to the proof of Theorem~\ref{label:intk}, it suffices only to consider the vertex $v_1$ instead of $v_i$. By Lemma~\ref{lem:int}, we have $d_{G^2}(v_1)\leq \Delta$, i.e., the number of vertices at distance at most $2$ from $v_1$ is at most $\Delta$. Since the set $\{0,m,2m,\ldots,\Delta m \}$ contains $\Delta+1$ distinct labels, there is an available label for $v_1$. If we assign such an available label to $v_1$, adjacent vertices receive labels at least $m=\max\{p,q\}\geq p$ apart, while the vertices at distance $2$ from each other get labels at least $m=\max\{p,q\} \geq q$ apart. Hence, $\lambda_{p,q}(G)\leq \max\{p,q\} \Delta$ as claimed.
\end{proof}

\begin{cor}
$\chi(G^2)= \Delta+1$ for every interval graph $G$.
\end{cor}

Our final move in this section is to prove the relevant claim of Theorem \ref{mainthm} for circular-arc graphs. Recall that circular-arc graphs are the intersection graphs of arcs on a circle and thus constitute a superclass of interval graphs. In \cite{Calamoneri-N}, Calamoneri \textit{et al.} considered the $L(p,q)$-labeling of a circular-arc graph by dividing it into two parts and labeling them seperately. In more detail, they partition the vertex set of a circular-arc graph $G$ into two sets $S$ and $C$ such that the subgraph $G[S]$ induced by $S$ is an interval graph, while $C$ induces a clique in $G$. Observe that this can be done by excluding a set $C$ of vertices corresponding to arcs having a common point $X$ on the circle (see Figure~\ref{fig:circular}). Obviously $C$ corresponds to a clique in $G$, while taking projection from the point $X$ onto $x$ or $y$ axis provides an interval representation from the remaining arcs corresponding to $S$. Hence, Calamoneri \textit{et al.}~\cite{Calamoneri-N} concluded that $G$ can be $L(p,q)$-labeled by labeling the subgraph induced by $S$ via Theorem~\ref{cala:int} and using additional labels for $C$. The upper bound proposed in \cite[Theorem $6$]{Calamoneri-N} is $\max\{p,2q\}\Delta+p\omega$ . 

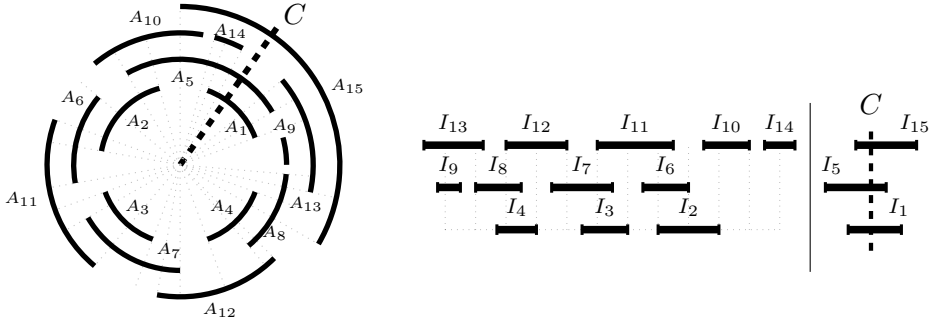
\begin{figure}[htb]
\centering
\parbox{5cm}{
\begin{tikzpicture}[scale=0.7]
\draw [line width=0.8mm,dashed](0,0) -- ++(55:3.3cm)node[pos=.9,above right] {$C$};
\draw [dotted,lightgray](0,0) -- ++(20:2.5cm);%
\draw [dotted,lightgray](0,0) -- ++(30:2.5cm); %
\draw [dotted,lightgray](0,0) -- ++(62:2.5cm); %
\draw [dotted,lightgray](0,0) -- ++(70:2.5cm);%
\draw [dotted,lightgray](0,0) -- ++(75:2.5cm);%
\draw [dotted,lightgray](0,0) -- ++(105:2.5cm);%
\draw [dotted,lightgray](0,0) -- ++(120:2.5cm); %
\draw [dotted,lightgray](0,0) -- ++(140:2.5cm);%
\draw [dotted,lightgray](0,0) -- ++(160:2.5cm);%
\draw [dotted,lightgray](0,0) -- ++(170:2.5cm);%
\draw [dotted,lightgray](0,0) -- ++(190:2.5cm);%
\draw [dotted,lightgray](0,0) -- ++(200:2.5cm);%
\draw [dotted,lightgray](0,0) -- ++(210:2.5cm);%
\draw [dotted,lightgray](0,0) -- ++(230:2.5cm);%
\draw [dotted,lightgray](0,0) -- ++(250:2.5cm);%
\draw [dotted,lightgray](0,0) -- ++(260:2.5cm);%
\draw [dotted,lightgray](0,0) -- ++(270:2.5cm);%
\draw [dotted,lightgray](0,0) -- ++(290:2.5cm);%
\draw [dotted,lightgray](0,0) -- ++(310:2.5cm);%
\draw [dotted,lightgray](0,0) -- ++(315:2.5cm);%
\draw [dotted,lightgray](0,0) -- ++(330:3cm);%
\draw [dotted,lightgray](0,0) -- ++(340:2.5cm);%
\draw [dotted,lightgray](0,0) -- ++(348:3cm);%
\draw [dotted,lightgray](0,0) -- ++(355:2.5cm);%
\draw [dotted,lightgray](0,0) -- ++(400:2.5cm);%
\draw [dotted,lightgray](0,0) -- ++(450:3cm);%
\draw[line width=0.7mm] ([shift=(20:1.5cm)]0,0) arc (20:70:1.5cm) node[midway,below] {\tiny{$A_1$}};
\draw[line width=0.7mm] ([shift=(105:1.5cm)]0,0) arc (105:170:1.5cm)node[near start,pos=.65,right] {\tiny{$A_2$}};
\draw[line width=0.7mm] ([shift=(200:1.5cm)]0,0) arc (200:250:1.5cm)node[near start,right] {\tiny{$A_3$}};
\draw[line width=0.7mm] ([shift=(290:1.5cm)]0,0) arc (290:340:1.5cm)node[near end,left] {\tiny{$A_4$}};
\draw[line width=0.7mm] ([shift=(30:2cm)]0,0) arc (30:120:2cm)node[pos=.65,below] {\tiny{$A_5$}};
\draw[line width=0.7mm] ([shift=(310:2cm)]0,0) arc (310:355:2cm)node[midway,below] {\tiny{$A_8$}};
\draw[line width=0.7mm] ([shift=(0:2cm)]0,0) arc (0:15:2cm)node[pos=.7,above] {\tiny{$A_9$}};
\draw[line width=0.7mm] ([shift=(140:2cm)]0,0) arc (140:190:2cm)node[pos=.01,left] {\tiny{$A_6$}};
\draw[line width=0.7mm] ([shift=(210:2cm)]0,0) arc (210:270:2cm)node[pos=.9,above] {\tiny{$A_7$}};
\draw[line width=0.7mm] ([shift=(348:2.5cm)]0,0) arc (348:400:2.5cm)node[midway, below] {};
\node [below] at (350:2.4cm) {\tiny{$A_{13}$}};
\draw[line width=0.7mm] ([shift=(80:2.5cm)]0,0) arc (80:130:2.5cm)node[midway,above] {\tiny{$A_{10}$}};
\draw[line width=0.7mm] ([shift=(160:2.5cm)]0,0) arc (160:230:2.5cm)node[midway,left] {\tiny{$A_{11}$}};
\draw[line width=0.7mm] ([shift=(260:2.5cm)]0,0) arc (260:315:2.5cm)node[midway,below] {\tiny{$A_{12}$}};
\draw[line width=0.7mm] ([shift=(330:3cm)]0,0) arc (330:450:3cm)node[midway,right] {\tiny{$A_{15}$}};
\draw[line width=0.7mm] ([shift=(62:2.5cm)]0,0) arc (62:75:2.5cm)node[midway,above] {};
\node [above] at (67:2.4cm) {\tiny{$A_{14}$}};
\end{tikzpicture}
\label{fig:diagram}}
\quad
\begin{minipage}{7cm}
\begin{tikzpicture}[xscale=0.4,yscale=0.56]	
\draw [help lines,dotted, ystep=18] (-1,0) grid(10,2);
		\draw [line width=1mm] (0,1) -- (1.5,1)node[midway,above] {\scriptsize{$I_{8}$}};
			\node [vert] at (0,1) {};
       		 \node [vert] at (1.5,1) {};
 		\draw [line width=1mm] (-1.25,1) -- (-0.5,1)node[midway,above] {\scriptsize{$I_{9}$}};
			\node [vert] at (-1.25,1) {};
       		 \node [vert] at (-0.5,1) {};
 		\draw [line width=1mm] (0.7,0) -- (2,0)node[midway,above] {\scriptsize{$I_4$}};
			\node [vert] at (0.7,0) {};
       		 \node [vert] at (2,0) {};
		\draw [line width=1mm] (1,2) -- (3,2)node[midway,above] {\scriptsize{$I_{12}$}};
			\node [vert] at (1,2) {};
       		 \node [vert] at (3,2) {};
		\draw [line width=1mm] (-1.7,2) -- (0.25,2)node[midway,above] {\scriptsize{$I_{13}$}};
			\node [vert] at (-1.7,2) {};
       		 \node [vert] at (0.25,2) {};
		\draw [line width=1mm] (2.5,1) -- (4.5,1)node[midway,above] {\scriptsize{$I_7$}};
			\node [vert] at (2.5,1) {};
       		 \node [vert] at (4.5,1) {};
		\draw [line width=1mm] (3.5,0) -- (5,0)node[midway,above] {\scriptsize{$I_3$}};
			\node [vert] at (3.5,0) {};
       		 \node [vert] at (5,0) {};
		\draw [line width=1mm] (4,2) -- (6.5,2)node[midway,above] {\scriptsize{$I_{11}$}};    
			\node [vert] at (4,2) {};
       		 \node [vert] at (6.5,2) {};   		
		\draw [line width=1mm] (5.5,1) -- (7,1)node[midway,above] {\scriptsize{$I_6$}};  
			\node [vert] at (5.5,1) {};
       		 \node [vert] at (7,1) {};
		\draw [line width=1mm] (6,0) -- (8,0)node[midway,above] {\scriptsize{$I_2$}}; 
			\node [vert] at (6,0) {};
       		 \node [vert] at (8,0) {};
		\draw [line width=1mm] (7.5,2) -- (9,2)node[midway,above] {\scriptsize{$I_{ 10}$}}; 	
			\node [vert] at (7.5,2) {};
       		 \node [vert] at (9,2) {};
		\draw [line width=1mm] (9.5,2) -- (10.5,2)node[midway,above] {\scriptsize{$I_{ 14}$}}; 	
			\node [vert] at (9.5,2) {};
       		 \node [vert] at (10.5,2) {};
 		\draw [line width=1mm] (11.5,1) -- (13.5,1)node[midway,above left] {\scriptsize{$I_{5}$}}; 	
			\node [vert] at (11.5,1) {};
       		 \node [vert] at (13.5,1) {};
		\draw [line width=1mm] (12.25,0) -- (14,0)node[midway,above right] {\scriptsize{$I_1$}}; 	
			\node [vert] at (12.25,0) {};
       		 \node [vert] at (14,0) {};       		 
		\draw [line width=1mm] (12.5,2) -- (14.5,2)node[midway,above right] {\scriptsize{$I_{15}$}};	
			\node [vert] at (12.5,2) {};
       		 \node [vert] at (14.5,2) {};   
		\draw [] (11,-1) -- (11,3); 
        \draw [ultra thick, dashed] (13,-0.5) -- (13,2.5);   
		\node [above] at (13,2.5) {$C$}; 	 
\end{tikzpicture}   
\end{minipage}
\vspace*{-4mm}
\caption{Partition of a circular-arc graph into an interval graph and a clique}
\label{fig:circular}
\end{figure}
\vspace*{-3mm}
\begin{rem}
It is natural to ask whether the upper bound of Theorem~\ref{label:int} can be used to improve the mentioned upper bound for circular-arc graphs. However, as one of the referees pointed out, the proof of Theorem 6 in \cite{Calamoneri-N} turns out to have a gap. For any distinct two vertices $u,v\in S$ with $d_{G[S]}(u,v)>2$, it may be the case that $u$ and $v$ have a mutual neighbor from $C$, implying that $d_G(u,v)=2$. But, an $L(p,q)$-labeling of $G[S]$ may assign even the same label for $u$ and $v$, although they must get labels with difference at least $q$ in an $L(p,q)$-labeling of $G$. Hence, an arbitrary $L(p,q)$-labeling of the induced subgraph $G[S]$ can not be directly extended to that of $G$. For instance, let $G$ be the graph which is obtained from a $8$-path $P_8$ by joining its end vertices with the vertices of a $P_2$ (see Figure~\ref{fig:ex-circ}). If we choose $C=\{x,y\}$ as in the figure, then $G[S]$ is isomorphic to $P_8$. For an $L(2,1)$-labeling of $G[S]$, the vertices $v_1,v_2,\ldots,v_8$ can be assigned with the labels in the order $(0, 2,4,6,1,3,5,0)$, as described in \cite[Lemma 3]{Calamoneri-N}. On the other hand, such a labeling is not admissible for an $L(2,1)$-labeling of $G$, since $d_G(v_1,v_8)=2$. (Here, we note that we set $\Delta:=\Delta(G)=3$ for the labeling of the vertices of $G[S]$. If the desired setting is $\Delta:=\Delta(G[S])$ for the $L(2,1)$-labeling of $G[S]$, then a similar argument can be repeated for the graph constructed from a $P_6$ in the same way.)
\end{rem}

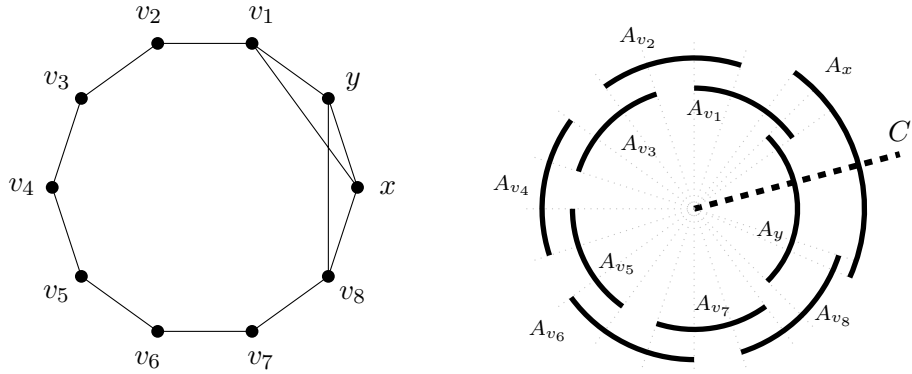
\begin{figure}[htb]
\centering     
\subfigure{
\begin{tikzpicture}[scale=.4]
  \node[
    regular polygon,
    regular polygon sides=10,
    minimum size=4cm,
  ] (a) {};
  \draw[black]
    (a.corner 8) \foreach \i in {9, 10, 1, 2, 3} { -- (a.corner \i) }
    (a.corner 3) \foreach \i in {4, ..., 8} { -- (a.corner \i) }
  ;
  \draw[black]
    (a.corner 1) -- (a.corner 9)
    (a.corner 8) -- (a.corner 10)
  ;
 \foreach \x in {1,2,...,8}
  \fill (a.corner \x) circle[radius=6pt] node[shift={(\x*360/10+35:0.4)}] {$v_{\x}$};
  \fill (a.corner 9) circle[radius=6pt] node[shift={(9*360/10+35:0.4)}] {$x$};
  \fill (a.corner 10) circle[radius=6pt] node[shift={(10*360/10+35:0.4)}] {$y$};
\end{tikzpicture}
}
\hspace*{.5cm}
\subfigure{
\begin{tikzpicture}[scale=0.8]
\draw [line width=0.8mm,dashed](0,0) -- ++(15:3.5cm)node[pos=1,above] {$C$};
\draw [dotted,lightgray](0,0) -- ++(-24:2.8cm);
\draw [dotted,lightgray](0,0) -- ++(36:2.8cm);
\draw [dotted,lightgray](0,0) -- ++(54:2.8cm);
\draw [dotted,lightgray](0,0) -- ++(72:2.8cm);
\draw [dotted,lightgray](0,0) -- ++(90:2.8cm);
\draw [dotted,lightgray](0,0) -- ++(108:2.8cm);
\draw [dotted,lightgray](0,0) -- ++(126:2.8cm);
\draw [dotted,lightgray](0,0) -- ++(144:2.8cm);
\draw [dotted,lightgray](0,0) -- ++(162:2.8cm);
\draw [dotted,lightgray](0,0) -- ++(180:2.8cm);
\draw [dotted,lightgray](0,0) -- ++(198:2.8cm);
\draw [dotted,lightgray](0,0) -- ++(216:2.8cm);
\draw [dotted,lightgray](0,0) -- ++(234:2.8cm);
\draw [dotted,lightgray](0,0) -- ++(252:2.8cm);
\draw [dotted,lightgray](0,0) -- ++(270:2.8cm);
\draw [dotted,lightgray](0,0) -- ++(288:2.8cm);
\draw [dotted,lightgray](0,0) -- ++(306:2.8cm);
\draw [dotted,lightgray](0,0) -- ++(315:2.8cm);
\draw [dotted,lightgray](0,0) -- ++(342:2.8cm);
\draw [dotted,lightgray](0,0) -- ++(405:2.8cm);
\draw[line width=0.7mm] ([shift=(36:2cm)]0,0) arc (36:90:2cm)node[pos=.9,below] {\scriptsize{$A_{v_1}$}};
\draw[line width=0.7mm] ([shift=(108:2cm)]0,0) arc (108:162:2cm)node[midway,below right] {\scriptsize{$A_{v_3}$}};
\draw[line width=0.7mm] ([shift=(180:2cm)]0,0) arc (180:234:2cm)node[pos=.5, right] {\scriptsize{$A_{v_5}$}};
\draw[line width=0.7mm] ([shift=(252:2cm)]0,0) arc (252:306:2cm)node[midway, above] {\scriptsize{$A_{v_7}$}};
\draw[line width=0.7mm] ([shift=(315:1.7cm)]0,0) arc (315:405:1.7cm)node[midway,below left] {\scriptsize{$A_{y}$}};
\draw[line width=0.7mm] ([shift=(-24:2.8cm)]0,0) arc (-24:54:2.8cm)node[pos=.9, above right] {\scriptsize{$A_{x}$}};
\draw[line width=0.7mm] ([shift=(72:2.5cm)]0,0) arc (72:126:2.5cm)node[midway,above left] {\scriptsize{$A_{v_2}$}};
\draw[line width=0.7mm] ([shift=(144:2.5cm)]0,0) arc (144:198:2.5cm)node[midway,left] {\scriptsize{$A_{v_4}$}};
\draw[line width=0.7mm] ([shift=(216:2.5cm)]0,0) arc (216:270:2.5cm)node[pos=.1,below left] {\scriptsize{$A_{v_6}$}};
\draw[line width=0.7mm] ([shift=(288:2.5cm)]0,0) arc (288:342:2.5cm)node[midway,right] {\scriptsize{$A_{v_8}$}};
\end{tikzpicture}
}
\vspace*{-4mm}
\caption{A graph and its circular-arc representation.}
\label{fig:ex-circ}
\end{figure}

However, we are able to establish an upper bound for the $L(p,q)$-labeling of circular-arc graphs. Our proof benefits the same technique, with an additional relabeling procedure on the subgraph $G[S]$ so that the newly created $L(p,q)$-labeling of $G[S]$ can be extended to an $L(p,q)$-labeling of the whole graph $G$. 

When the circular-arc representation of a circular-arc graph is considered, we use clockwise direction for traversing the circle. We denote by $A_v$, the arc corresponding to a vertex $v$ and use the notations $s(A_v)$ and $t(A_v)$ for the beginning and ending point of the arc $A_v$ in a clockwise traversal, respectively. In other words, an arc $A_v$ is defined by traversing the circle from $s(A_v)$ to $t(A_v)$ in the clockwise direction.

\begin{theorem}\label{circarc}
A circular-arc graph $G$ with at least one edge has an $L(p,q)$-labeling with a span at most $3\max\{p,q\}\Delta+p$.
\end{theorem}

\begin{proof}
Let $\{A_v \colon \, v\in V\}$ be the set of arcs in a circular-arc representation of $G$. We first note that if there is an arc covering the whole circle, then it is clear that $|V|=\Delta+1$. If there are two arcs $A_u$ and $A_v$ covering the circle, then $|V|\leq 2\Delta$, since $V=N_G(u)\cup N_G(v)$. In both cases, an $L(p,q)$-labeling of $G$ with a span at most $3\max\{p,q\}\Delta+p$ can be trivially achieved by assigning each vertex of $G$ with a distinct label from the set $\{0,m,2m,\ldots,3\Delta m\}$, where $m=\max\{p,q\}$. We may therefore suppose that there is neither an arc nor two arcs covering the circle in the circular-arc representation of $G$.

Choose a subset $C$ of $V$ such that corresponding arcs have a common point on the circle. Let $r$ be the size of $C$. Since $G$ has at least one edge, we may assume without loss of generality that $C$ is chosen such that $r\geq 2$. Recall that if we set $S=V\backslash C$, then the subgraph $G[S]$ is an interval graph. Pick an $L(p,q)$-labeling of $G[S]$ with a span $\lambda \leq m\Delta$ provided from Theorem~\ref{label:int}, such that the greatest label used is $\lambda$.

Next, we consider the set
$$S':=\{u\in S \colon \, d_{G[S]}(u,v)>2 \text{ and } d_G(u,v)=2 \text{ for some }v\in S\}.$$

If $S'$ is nonempty, then we relabel any $(|S'|-1)$ of the vertices in $S'$, so that any two vertices of $S'$ receive labels which differ by at least $m$. Subsequently, we extend the current $L(p,q)$-labeling of $G[S]$ to an $L(p,q)$-labeling of $G$, by assigning additional labels to the vertices of $C$. It is clear that if $S'=\emptyset$, then we may proceed the proof without a relabeling process and assign the vertices of $C$ with the same additional labels used in the case when $S'\neq\emptyset$. Therefore, it is sufficient to consider only the case $S'\neq\emptyset$. Then, $|S'|\geq 2$ by the choice of $S'$.

Among all the vertices in $C$, let $x,y\in C$ be such that $s(A_x)$ appears first and $t(A_y)$ appears last, while traversing the arcs corresponding to the vertices of $C$ (see, for instance, the arcs $A_5$ and $A_{15}$ in Figure~\ref{fig:circular}). We now have two cases depending on whether $x=y$ or not.
\vspace*{2mm}

\textit{Case 1.} $x\neq y$. It is easy to see that if $u\in S'$, then $u$ is adjacent to at least one of $x$ and $y$. Therefore we have $S'\subseteq N_G(x)\cup N_G(y)$. If we set $k:=m_G(x,y)$, then it follows that $|S'|\leq d_G(x)+d_G(y)-k \leq 2\Delta-k$. Note that any two vertices in $S'$ may also be adjacent. Therefore, we use a distinct label from the set $\{\lambda+m, \lambda+2m,\ldots, \lambda+(2\Delta-k-1)m\}$ for each vertex except one in $S'$. 

In the final step, we label the vertices in $C$. Let $\lambda'=\lambda+(2\Delta-k-1)m$. Here, we remark that the label $\lambda'+p$ may not be admissible for a vertex $c\in C$, because it may be the case that $p<q$ and that $c$ may be at distance $2$ from a vertex (from $S'$) which have the label $\lambda'$. Therefore we assign the labels  $\lambda'+m, \lambda'+m+p, \ldots, \lambda'+m+(r-1)p$ to the vertices of $C$, in an arbitrary order. Such a labeling clearly ensures that the vertices of $C$ receive labels with difference at least $p$ amongst each other. Hence we obtain an $L(p,q)$-labeling of $G$. Since each of the vertices in $C\backslash\{x,y\}$ (if any) is a common neighbor of $x$ and $y$, we have $|C|=r\leq 2+k$. Then, the largest label we use is at most
\begin{align*}
\lambda'+m+(r-1)p & =\lambda+(2\Delta-k-1)m+m+(r-1)p \\
& \leq 3m\Delta-km+(k+1)p=3m\Delta-k(m-p)+p \\ 
& \leq 3m\Delta+p=3\max\{p,q\}\Delta+p.
\end{align*}

\textit{Case 2.} $x=y$. In such a case, let $t$ be the number of the neighbors of $x$ in $C$. Note that $t\geq 1$ by the choice of $C$, and $|C|=t+1$. We may further observe that $|S'|\leq d_G(x)-t\leq \Delta-t$, since all the vertices in $S'$ are adjacent to $x$. Then we use a distinct label from the set $\{\lambda+m, \lambda+2m,\ldots, \lambda+(\Delta-t-1)m\}$ for each vertex except one in $S'$, and let $\lambda'=\lambda+(\Delta-t-1)m$. Finally, we label the vertices of $C$. Similar to the previous case, we may assign the labels  $\lambda'+m, \lambda'+m+p, \ldots, \lambda'+m+tp$ to the vertices of $C$, in an arbitrary order. Thus, the largest label we use is at most
\begin{align*}
\lambda'+m+tp & =\lambda+(\Delta-t-1)m+m+tp\\
&  \leq 2m\Delta-tm-m+m+tp \\
& =2m\Delta-t(m-p)  \leq 2m\Delta-(m-p) \\  
& = m(2\Delta-1)+p=\max\{p,q\}(2\Delta-1)+p.
\end{align*}

Hence the result follows, since $\max\{p,q\}(2\Delta-1)+p  \leq 3\max\{p,q\}\Delta+p$ for each values of $p$, $q$ and $\Delta$.
\end{proof}

Theorem~\ref{circarc} has an immediate consequence that $\lambda_{2,1}(G) \leq 6\Delta+2$ holds for a circular-arc graph $G$. On the other hand, we have the following upper bound on the chromatic number of the square of circular-arc graphs. 

\begin{cor}\label{thm:square-circular}
$\chi(G^2) \leq 3\Delta+2$ for every circular-arc graph $G$.
\end{cor}


\section{Permutation (Interval Containment) Graphs}\label{perm}

Permutation graphs constitute a well-studied graph class. Even though they admit various representations, we here consider the interval representation of these graphs. In that language, a permutation graph is a graph with interval representation on the real line such that two vertices are adjacent if and only if the corresponding intervals are comparable with respect to the inclusion.

Bodlaender \textit{et al.}. \cite{Bodlaender} gave an $O(n\Delta)$ time algorithm to show that $\lambda_{2,1}(G) \leq 5\Delta-2 $ for every permutation graph $G$. Moreover, Paul \textit{et al.}~\cite{Paul} improve this bound to $\max\{4\Delta-2, \, 5\Delta-8\}$ by appealing to the algorithm given in \cite{chang1996}. Notice that these results were obtained by considering matching diagrams of permutation graphs. We here provide an upper bound in the general case of $L(p,q)$-labeling of permutation graphs and improve the previously known results on $L(2,1)$-labeling.

\begin{theorem}\label{thm:perm}
If an ordering $\sigma=(v_1,v_2,\ldots,v_n)$ of the vertices of a permutation graph $G$ such that $r(v_1)\leq r(v_2) \leq \ldots \leq r(v_n)$ in an interval representation of $G$ is given as the input, then Algorithm~\ref{algo:alg1} produces an $L(p,q)$-labeling of $G$ with a span at most $2(p+q-1)\Delta-2q+1$.
\end{theorem}

\begin{proof}
Assume that the vertices $v_{i+1},v_{i+2},\ldots,v_n$ of $G$ have been labeled by the Algorithm~\ref{algo:alg1}. Since all these labeled vertices corresponds to intervals with right endpoints no less than $r(v_i)$, counting the forbidden labels for the vertex $v_i$ is equivalent to counting the forbidden labels for the vertex $v_1$. Therefore, once again, consider the vertex $v_1$ instead of $v_i$ as in the previous proofs. Now we may observe that if $v_1$ has no neighbor $w$ with $r(v_1)<r(w)$, then $v_1$ has no second neighbor. In such a case, there are at most $(2p-1)d_G(v_1)\leq (2p-1)\Delta $ forbidden labels for $v_1$. Therefore we may further assume that $v_1$ has at least one neighbor whose interval has a larger right endpoint. Let $v_{i_1},v_{i_2},\ldots,v_{i_s}$ be such neighbors of $v_1$ in the increasing order with respect to their right endpoints. By the choice of $v_1$, we have $I_{v_1}\subseteq I_{v_{i_k}}$ for each $1\leq k \leq s $.

\begin{claim*} 
$N_G(v_{i_j})\setminus N_G[v_1]\subseteq N_G(v_{i_s})$ for every $j$ with $1\leq j<s$.
\end{claim*}

\begin{proof}  Let $v_{i_j}$ be a neighbor of $v_1$ such that $1\leq j<s$, and pick a vertex $u\in N_G(v_{i_j})\setminus N_G[v_1]$. Since $r(v_1)\leq r(u)$ by the choice of $v_1$, we first have $l(v_1)<l(u)$. Secondly, we need to show that $r(u)\leq r(v_{i_s})$. Because of the facts $ I_{v_1}\subseteq I_{v_{i_j}}$ and $ I_{v_1}\not\subset I_u$, we have $ I_u\subseteq I_{v_{i_j}}$, which means that $r(u)\leq r(v_{i_j})$. Since we also have $r(v_{i_j})< r(v_{i_s})$, we conclude that $r(u)< r(v_{i_s})$. Combining these, we obtain $l(v_{i_s})\leq l(v_1)<l(u)<r(u)< r(v_{i_s})$, which yields $uv_{i_s}\in E$. This completes the proof of the claim.
\end{proof}

We remark that the above claim implies that all the vertices at distance $2$ from $v_1$ are adjacent to $v_{i_s}$, which means that the number of second neighbors of $v_1$ is at most $d_G(v_{i_s})-1\leq \Delta-1$ (see Figure~\ref{fig:perm}).

\begin{figure}[htb]
\centering    
\begin{tikzpicture}[yscale=0.5,xscale=0.6]
\draw [help lines,dotted, ystep=22] (0,0) grid(9,5);
	\node [left] at (1.5,4.9) {$I_{v_1}$};
		\draw [line width=1mm] (0,1) -- (6,1); 
			\node [vert] at (0,1) {};
       		 \node [vert] at (6,1) {};
		\draw [line width=1mm] (0.5,3) -- (4,3);  
			\node [vert] at (0.5,3) {};
       		 \node [vert] at (4,3) {};
 		\draw [line width=1mm] (1,0) -- (7.5,0);
			\node [vert] at (1,0) {};
       		 \node [vert] at (7.5,0) {};
		\draw [line width=1mm] (1.5,5) -- (3,5); 
		\draw [pattern=checkerboard, pattern color=lightgray] (1.5,4.9) rectangle (3,5.1);
			\node [vert] at (1.5,5) {};
       		 \node [vert] at (3,5) {};
		\draw [line width=1mm,lightgray] (2,2) -- (3.5,2);
			\node [vertl] at (2,2) {};
       		 \node [vertl] at (3.5,2) {};
		\draw [line width=1mm,lightgray] (2.5,4) -- (5.5,4);
			\node [vertl] at (2.5,4) {};
       		 \node [vertl] at (5.5,4) {};
		\draw [gray,fill=white,pattern=north east lines, pattern color=gray] (3.75,4.9) rectangle (9,5.1);
			\node [vertg] at (3.75,5) {};
       		 \node [vertg] at (9,5) {};       		 
		\draw [gray,fill=white,pattern=north east lines, pattern color=gray] (4.5,2.9) rectangle (8,3.1);
			\node [vertg] at (4.5,3) {};
       		 \node [vertg] at (8,3) {};
		\draw [line width=1mm,lightgray] (5,2) -- (6.5,2);
			\node [vertl] at (5,2) {};
       		 \node [vertl] at (6.5,2) {};
		\draw [gray,fill=white,pattern=north east lines, pattern color=gray] (7,0.9) rectangle (8.5,1.1);
			\node [vertg] at (7,1) {};
       		 \node [vertg] at (8.5,1) {};
       \draw [line width=1mm] (11,4) -- (12,4);
       \draw [line width=1mm, lightgray] (11,2) -- (12,2);
 		\node [right] at (12,4) {(first) neighbors of $v_1$};
 		\node [right] at (12,2) {second neighbors of $v_1$};
\end{tikzpicture} 
\caption{An illustration for the proof of Theorem~\ref{thm:perm}.}
\label{fig:perm}
\end{figure}
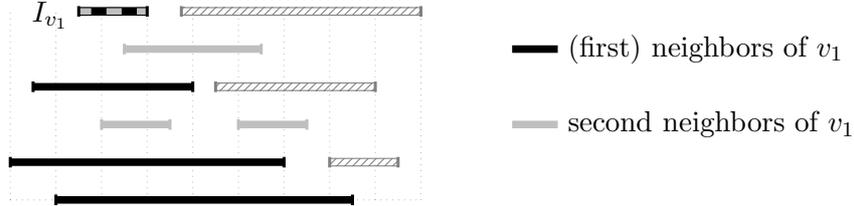

We now compute the number of forbidden labels for $v_1$. For each neighbor of $v_1$, there are $2(p-1)+1=2p-1$ labels that are forbidden for $v_1$, and for each second neighbor of $v_1$, there are $2(q-1)+1=2q-1$ labels that are forbidden for $v_1$. So, in total, we have at most 
\begin{align*}
&(2p-1)d_G(v_1)+(2q-1)(d_G(v_{i_s})-1)\\
&\leq (2p-1)\Delta+(2q-1)(\Delta-1)\\
&\leq 2(p+q-1)\Delta-2q+1 
\end{align*}
labels which are forbidden for $v_1$. This completes the proof.
\end{proof}

\begin{cor}
 $\lambda_{2,1}(G) \leq 4\Delta-1 $ for every permutation graph $G$. Thus, Algorithm~\ref{algo:alg1} for $L(2,1)$-labeling of permutation graphs is a $4$-approximation algorithm.
\end{cor}

\begin{cor}
$\chi(G^2) \leq 2\Delta $ for every permutation graph $G$.
\end{cor}

\section{Cointerval Graphs (Comparability Graphs of Interval Orders)}\label{coint}
In this section, we first observe the equivalence between cointerval graphs and the comparability graphs of interval orders. Then, having the advantages of such an equivalence, we complete the proof of Theorem~\ref{mainthm} by verifying the claimed upper bound on $\lambda_{p,q}$ for cointerval graphs. We refer readers to \cite{trotter2001} for the terminology of partially ordered sets. 

In particular, the set of minimal elements of $P$ is denoted by $\textnormal{Min}(P)$. The \emph{comparability graph} of a poset $P=(X,\leq_P)$ is the graph $G(P)$ (or simply $G$) on the same set $X$ such that two vertices $x,y\in X$ are adjacent in $G(P)$ if and only if $x$ and $y$ are comparable in $P$. A poset $P=(X,\leq_p)$ is said to be an \emph{interval order} if there exists an interval representation $\mathcal{I}_P=\{I_x: \, x\in X\}$ of $P$ such that $x<_Py$ if and only if $r(I_x)<l(I_y)$. A graph is called a \emph{cointerval graph} if its complement is an interval graph. $2K_2$ denotes the graph consisting of two disjoint edges. It is known (see, \cite{peled1995}) that a graph $G$ is a cointerval graph if and only if it is a comparability graph with no induced $2K_2$. 

One may also observe that a cointerval graph $G$ admits an interval representation such that $uv\in E(G)$ if and only if $I_u\cap I_v= \emptyset$, since $\overline{G}$ is an interval graph. Therefore it is straightforward that $G$ is the comparability graph of an interval order $P$ on the set $V(G)$ such that $u<_P v$ if $r(I_u)<l(I_v)$ for $u,v\in V(G)$. Conversely, if $G$ is the comparability graph of an interval order $P$, then $P$ has an interval representation such that two elements are comparable in $P$ if and only if corresponding intervals do not intersect. It then follows that $G$ is a cointerval graph. Thus, we have the following equivalence.

\begin{fact}
A graph $G$ is a cointerval graph if and only if $G$ is the comparability graph of an interval order.
\end{fact}

\begin{cor}
Let $G$ be a graph. Then the following statements are equivalent:
\begin{enumerate}[label=\normalfont(\roman*),itemsep=0em]
\item $G$ is a cointerval graph.
\item $G$ is the comparability graph of an interval order.
\item $G$ is a $2K_2$-free comparability graph.
\end{enumerate}
\end{cor}

We are now ready to prove the relevant claim of Theorem~\ref{mainthm} for cointerval graphs. We note that by an interval representation of a cointerval graph $G$, we mean the interval representation of the corresponding interval order whose comparability graph is $G$. 

\begin{theorem}\label{thm:coint}
Algorithm~\ref{algo:alg2} produces an $L(p,q)$-labeling of a cointerval graph $G$ with a span at most $(2p-1)\Delta+(2q-1)(\mu-1)\leq 2(p+q-1)\Delta-2q+1$.
\end{theorem}

\begin{proof}
Let $G$ be a cointerval graph, that is, the comparability graph of an interval order $P$, and let $v_1,v_2,\ldots,v_n$ be its vertices such that $r(v_1)\leq r(v_2) \leq \ldots \leq r(v_n)$ in the interval representation. As in the proof of Theorem~\ref{label:intk}, we assume that $\Delta(G)\geq 2$ and $\mu(G)\geq 1$. Suppose that the non-leaf vertices from the set $\{v_{i+1},v_{i+2},\ldots,v_n\}$ of $G$ have been labeled by the first stage of Algorithm~\ref{algo:alg2}. As earlier, consider the vertex $v_1$ instead of $v_i$. Clearly, $v_1$ is a minimal element in $P$. Observe also that we have $v_1<_P u$, and thus $v_1u\in E$ for every non-minimal element $u$ in $P$ by the choice of $v_1$. If $v_1$ is the only minimal element of $P$, then this means that $v_1$ has no second neighbor. In such a case, the number of forbidden labels for $v_1$ is at most $(2p-1)d(v_1)\leq (2p-1)\Delta$. So, we may assume that $P$ has at least two minimal elements. A similar argument applies if $v_1$ has no labeled second neighbor. Thus we may further assume that $v_1$ has at least one labeled second neighbor, which implies that $\textnormal{Min}(P)\cap S\neq\emptyset$. Let $w$ be a minimal element such that $r(w)=\max\{r(u): \, u\in \textnormal{Min}(P)\cap S\}$. Since $d_G(w)\geq 2$, the vertex $w$ has at least two neighbors in $G$, say $a$ and $b$, so that the comparabilities  $w<_Pa$ and $w<_Pb$ hold. Observe that the set $N_G(a)\cap N_G(b)$ contains all the labeled minimal elements (labeled second neighbors of $v_1$). However, this implies that the number of labeled second neigbors of $v_1$ is at most $m_G(a,b)-1$. Note that the set $N_G(a)\cap N_G(b)$ does not contain unlabeled second neighbors (if any) of $v_1$, as they are leaf vertices of $G$ (see Figure~\ref{fig:coint}).

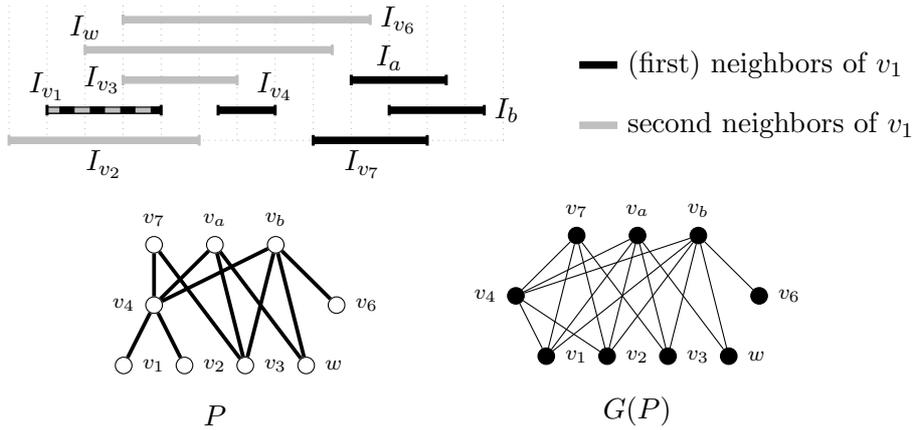
\begin{figure}[htb]
\centering     
\begin{tikzpicture}[yscale=0.4,xscale=0.5]
\draw [help lines,dotted, ystep=22] (0,0) grid(13,4.5);
		\node [above] at (1,1) {$I_{v_1}$};
		\node [right] at (9.5,4) {$I_{v_6}$};
		\node [above] at (2,3) {$I_{w}$};
		\node [right] at (12.5,1) {$I_{b}$};		
		\node [above] at (10,2) {$I_{a}$};
		\node [left] at (3.2,2) {$I_{v_3}$};
		\node [below] at (2.5,0) {$I_{v_2}$};
		\node [above] at (7,1) {$I_{v_4}$};		
		\node [below] at (9.3,0) {$I_{v_7}$};
		\draw [line width=1mm,lightgray] (0,0) -- (5,0); 
			\node [vertl] at (0,0) {};
       		 \node [vertl] at (5,0) {};
		\draw [line width=1mm] (1,1) -- (4,1);
		\draw [pattern=checkerboard, pattern color=lightgray] (1,0.9) rectangle (4,1.1);
			\node [vert] at (1,1) {};
       		 \node [vert] at (4,1) {};
 		\draw [line width=1mm,lightgray] (2,3) -- (8.5,3); 
			\node [vertl] at (2,3) {};
       		 \node [vertl] at (8.5,3) {};
 		\draw [line width=1mm,lightgray] (3,4) -- (9.5,4);
			\node [vertl] at (3,4) {};
       		 \node [vertl] at (9.5,4) {};
		\draw [line width=1mm,lightgray] (3,2) -- (6,2); 
			\node [vertl] at (3,2) {};
       		 \node [vertl] at (6,2) {};
		\draw [line width=1mm] (5.5,1) -- (7,1);  
			\node [vert] at (5.5,1) {};
       		 \node [vert] at (7,1) {};
		\draw [line width=1mm] (8,0) -- (11,0);
			\node [vert] at (8,0) {};
       		 \node [vert] at (11,0) {};
		\draw [line width=1mm] (10,1) -- (12.5,1); 
			\node [vert] at (10,1) {};
       		 \node [vert] at (12.5,1) {};
		\draw [line width=1mm] (9,2) -- (11.5,2);
			\node [vert] at (9,2) {};
       		 \node [vert] at (11.5,2) {};
       \draw [line width=1mm] (15,2.5) -- (16,2.5);
       \draw [line width=1mm, lightgray] (15,0.5) -- (16,0.5);
 		\node [right] at (16,2.5) {(first) neighbors of $v_1$};
 		\node [right] at (16,0.5) {second neighbors of $v_1$};
\end{tikzpicture} 
\vspace*{7mm}
\subfigure{
\begin{tikzpicture}[scale=.8]
\node [nodel] at (0,0) (v1) [label=right:\scriptsize{$v_1$}]  {};
\node [nodel] at (1,0) (v2) [label=right:\scriptsize{$v_2$}] {};
\node [nodel] at (2,0) (v3) [label=right:\scriptsize{$v_3$}] {};
\node [nodel]at (3,0) (v5) [label=right:\scriptsize{$w$}] {};
\node [nodel] at (0.5,1) (v4) [label=left:\scriptsize{$v_4$}] {}
	edge [line width=0.5mm] (v1)
	edge [line width=0.5mm] (v2);
\node [nodel]at (0.5,2) (v6) [label=above:\scriptsize{$v_7$}] {}
	edge [line width=0.5mm] (v4)
	edge [line width=0.5mm] (v3);
\node [nodel]at (1.5,2) (v7) [label=above:\scriptsize{$v_a$}] {}
	edge [line width=0.5mm] (v4)
    edge [line width=0.5mm] (v5)
	edge [line width=0.5mm] (v3);
\node [nodel] at (2.5,2) (v8) [label=above:\scriptsize{$v_b$}] {}
	edge [line width=0.5mm] (v4)
    edge [line width=0.5mm] (v5)
	edge [line width=0.5mm] (v3);
\node [nodel] at (3.5,1) (v9) [label=right:\scriptsize{$v_6$}] {}
	edge [line width=0.5mm] (v8);
	\node [below] at (1.5,-0.5) {$P$};
\end{tikzpicture} 
}
\hspace*{.5cm}
\subfigure{
\begin{tikzpicture}[scale=.8]
\node [nod2] at (0,0) (v1) [label=right:\scriptsize{$v_1$}]  {};
\node [nod2] at (1,0) (v2) [label=right:\scriptsize{$v_2$}] {};
\node [nod2] at (2,0) (v3) [label=right:\scriptsize{$v_3$}] {};
\node [nod2]at (3,0) (v5) [label=right:\scriptsize{$w$}] {};
\node [nod2]at (-0.5,1) (v4) [label=left:\scriptsize{$v_4$}] {}
	edge [] (v1)
	edge [] (v2);
\node [nod2]at (0.5,2) (v6) [label=above:\scriptsize{$v_7$}] {}
	edge [] (v1)
	edge [] (v2)
	edge [] (v4)
	edge [] (v3);
\node [nod2]at (1.5,2) (v7) [label=above:\scriptsize{$v_a$}] {}
	edge [] (v1)
	edge [] (v2)
	edge [] (v4)
    edge [] (v5)
	edge [] (v3);
\node [nod2]at (2.5,2) (v8) [label=above:\scriptsize{$v_b$}] {}
	edge [] (v1)
	edge [] (v2)
	edge [] (v4)
    edge [] (v5)
	edge [] (v3);
\node [nod2]at (3.5,1) (v9) [label=right:\scriptsize{$v_6$}] {}
	edge [] (v8);
\node [below] at (1.5,-0.5) {$G(P)$};
\end{tikzpicture} 
}
\vspace*{-6mm}
\caption{An illustration for the proof of Theorem~\ref{thm:coint}.}
\label{fig:coint}
\end{figure}

For each labeled neighbor of $v_1$, there are $2p-1$ labels that are forbidden for $v_1$; while for each labeled second neighbors of $v_1$, there are $2q-1$ labels that are forbidden for $v_1$. Therefore, we have totally at most 
$$(2p-1)d(v_1)+(2q-1)(m_G(a,b)-1)\leq (2p-1)\Delta+(2q-1)(\mu-1)$$
labels which are forbidden for $v_1$. This means that there exists an available label from the set $\{0,1,2,\ldots,(2p-1)\Delta+(2q-1)(\mu-1)\}$ for $v_1$. 

The proof for the second stage of the algorithm can be proceeded in the same way as in the proof of Theorem~\ref{label:intk}, independently from the ordering of leaf vertices. This completes the proof.
\end{proof}

As a result of Theorem~\ref{thm:coint}, the bound $\lambda_{2,1}(G) \leq 3\Delta+\mu-1\leq 4\Delta-1 $ holds for cointerval graphs. On the other hand,  Theorem~\ref{thm:coint} also yields the following  upper bound on the chromatic number of the square of cointerval graphs. 

\begin{cor}
$\chi(G^2) \leq \Delta+\mu $ for every cointerval graph $G$.
\end{cor}

\section{Concluding Remarks}

We have studied the $L(p,q)$-labeling problem on graphs with interval and circular-arc representations. We have provided upper bounds on $\lambda_{p,q}$ of interval $k$-graphs, circular-arc, permutation and cointerval graphs. In order to achieve that, we have performed simple yet natural greedy algorithms. While performing these algorithms, we have utilized linear ordering of the vertices of these graphs obtained from their interval representations on $\mathbb{R}$. 
The upper bounds given for interval $k$-graphs, permutation graphs and cointerval graphs are the first known results, while the upper bound for interval graphs refine the previous results, in the general case of $L(p,q)$-labeling. On the other hand, we improve the best known upper bound on $\lambda_{2,1}$ of permutation graphs. All the upper bounds we obtained for $L(2,1)$-labeling of these graphs are linear in terms of $\Delta$. As a consequence of the labeling we have carried out, we have obtained tight upper bounds on the chromatic number of the square of graphs with interval representations. Note that the complete bipartite graph $K_{r,r}$ for each $r\geq 2$ attains the upper bounds, as it is contained in interval $k$-graphs, permutation graphs and cointerval graphs.

\begin{center}
\textbf{Acknowledgments}
\end{center}

\noindent
We would like to thank Prof. Yusuf Civan for his invaluable comments and generous encouragement. We are grateful to the anonymous referees for their careful reading and constructive comments which greatly improved the manuscript. They pointed out several errors in the earlier versions of this paper and have contributed significantly to the overall improvement.

We also would like to thank one of the referees for pointing out that the proof of a result on circular-arc graphs in an earlier version contained a gap and that the gap originally appears in the proof of Theorem 6 of \cite{Calamoneri-N}, which our proof had relied on (and the proof of Theorem~\ref{circarc} is also based on).


\end{document}